\documentclass[12pt]{amsart}
\usepackage {amscd}
\usepackage[german,english]{babel}
\usepackage {amsfonts}
\usepackage[all]{xy}
\usepackage{amssymb,amsmath, amsthm,latexsym,verbatim}
\usepackage{a4wide}
\usepackage[hypertex]{hyperref}\usepackage{enumitem}
\newcommand{\XX}{\overline{X}}
\newcommand{\PX}{\partial\overline{X}}
\newcommand{\C}{\mathbb{C}}
\newcommand{\R}{\mathbb{R}}
\newcommand{\Z}{\mathbb{Z}}
\newcommand{\Q}{\mathbb{Q}}

\newcommand{\N}{\mathbb{N}}

\newcommand{\aL}{\mathfrak{a}}

\newcommand{\Aut}{\operatorname{Aut}}

\newcommand{\Id}{\operatorname{Id}}

\newcommand{\Hom}{\operatorname{Hom}}

\newcommand{\rank}{\operatorname{rank}}
\newcommand{\SU}{\operatorname{SU}}

\newcommand{\Sl}{\operatorname{SL}}

\newcommand{\PSL}{\operatorname{PSL}}
\newcommand{\SL}{\operatorname{SL}}

\newcommand{\pr}{\operatorname{pr}}

\newcommand{\Symm}{\operatorname{Sym}}
\newcommand{\vol}{\operatorname{vol}}
\newcommand{\Tr}{\operatorname{Tr}}

\newcommand{\Real}{\operatorname{Re}}

\newtheorem {thrm}{Theorem}[section]
\newtheorem {prop}[thrm] {Proposition}
\newtheorem {lem}[thrm] {Lemma}

\newtheorem {kor}[thrm]{Corollary}
\theoremstyle{definition}

\theoremstyle{remark}
\newtheorem {bmrk}[thrm] {Remark}

\begin{document}
\title[]{Exponential growth of homological torsion for towers of
congruence subgroups of Bianchi groups}

\author{Jonathan Pfaff}
\address{Universit\"at Bonn\\
Mathematisches Institut\\
Endenicher Alle 60\\
D - 53115 Bonn, Germany}
\email{pfaff@math.uni-bonn.de}
\begin{abstract}
In this paper we prove that for suitable sequences of congruence subgroups of
Bianchi 
groups, including the standard exhaustive sequences of a congruence subgroup,
and
even symmetric powers of
the standard 
representation of $\Sl_2(\C)$ the size of the torsion part in the first
homology 
grows exponentially. This extends results of Bergeron and Venkatesh to a 
case of non-uniform lattices. 
\end{abstract}

\maketitle
\section{Introduction}

The aim of the present paper is to extend some recent results obtained by
Bergeron and Venkatesh 
for cocompact arithmetic groups to the case of Bianchi groups which are not
cocompact. 
Thus we shall firstly recall some of the results of Bergeron and
Venkatesh. 
Let $\mathbf{G}$ be a semisimple connected linear algebraic group defined over
$\Q$ and let 
$\Gamma$ be a congruence subgroup of $\mathbf{G}(\mathbb{Q})$. Assume that
$\Gamma$ is cocompact or equivalently that $\mathbf{G}$ is anisotropic over
$\Q$. 
Let $G:=\mathbf{G}(\R)$ be the real points of
$\mathbf{G}$ and let 
$K$ be a maximal compact subgroup of $G$. Then one can form the globally
symmetric
space $\widetilde{X}:=G/K$ and the 
locally symmetric space $X:=\Gamma\backslash\widetilde{X}$. Let
$\delta(\widetilde{X}):=\rank_{\C}(G)-\rank_{\C}(K)$

Now let $L$ be an arithmetic $\Gamma$-module. This means that there exists a
$\Q$-rational representation $\rho$ of $\mathbf{G}$ on a finite-dimensional
$\Q$-vector space $V_\Q$ 
such that $L$ is a $\Z$-lattice in $V_\Q$ which is invariant under $\Gamma$.
Restricting 
$\rho$ to $\Gamma$, one can form the homology groups $H_q(\Gamma,L)$. These are
finitely generated abelian groups 
and one lets $H_q(\Gamma,L)_{tors}$ be their torsion part. If
$\delta(\widetilde{X})=1$ and $\widetilde{X}$ is odd-dimensional, for 
arithmetic reasons one expects these groups to be large, see \cite{BV} and also 
the recent book of Calegari and Venkatesh \cite{CV}. Bergeron
and 
Venkatesh verified this conjecture in the following way: If the representation
$\rho_\R$ of $G$ on $V_\Q\otimes_\Q\R$ associated to $\rho$ is strongly
acyclic, which means that it admits a uniform spectral gap near zero (see
\cite[page 15]{BV}), then for every sequence
$\Gamma_i$ of finite 
index (congruence) subgroups of $\Gamma$ for which the injectivity radius goes
to infinity, they
proved that 
\begin{align} \label{eqBV}
\liminf_{i\to\infty}\sum_{q\equiv d
(2)}\frac{\log|H_q(\Gamma_i,L)_{tors}|}{[\Gamma:\Gamma_i]}
\ge c_{\widetilde{X}}(\rho)\vol(X),
\end{align}
see \cite[Theorem
1.4]{BV},
where $d=\frac{\dim\widetilde{X}-1}{2}$. Here $c_{\widetilde{X}}(\rho)$ is a
strictly positive constant which depends
only on the globally symmetric space $\widetilde{X}$. More precisely, if X is
any (compact) manifold 
of the form $X=\Gamma\backslash\widetilde{X}$, where $\Gamma$ is a discrete,
torsion-free subgroup of 
$G$, and if one denotes by
$\log{T_X^{(2)}(\rho)}$ 
the $L^2$-torsion of $X$ with respect to the flat bundle defined by the
restricion of $\rho_\R$ 
to $\Gamma$, then one has
\begin{align}\label{L2tors}
\log{T_X^{(2)}(\rho)}=(-1)^{d}c_{\widetilde{X}}(\rho)\vol(X). 
\end{align}

Furthermore, Bergeron and Venkatesh
conjectured that in the asymptotic formula \eqref{eqBV} only 
the $d$-th part contributes to the growth of the torsion in the limit, i.e. they
conjectured that for every $q$ the limit 
\begin{align*}
\lim_{i\to\infty}\frac{\log|H_q(\Gamma_i,L)_{tors}|}{[\Gamma:\Gamma_i]}
\end{align*}
exists, that for $q\neq d$ it is zero and that for $q=d$ it equals
$c_{\widetilde{X}}(\rho)\vol(X)$, see
\cite[Conjecture 1.3]{BV}.
This conjecture has been verified by Bergeron and Venkatesh in \cite{BV} for
cocompact
arithmetic subgroups of $\Sl_2(\C)$ and for strongly acyclic representations.
\newline

In the present paper we prove a modified analog of the result \eqref{eqBV} of
Bergeron and Venkatesh for congruence subgroups of Bianchi groups, which are
non-uniform
lattices of $\Sl_2(\C)$, and for certain
strongly acyclic representations $\rho$. The group $\SU(2)$ is a maximal 
compact subgroup of $\Sl_2(\C)$ and the globally symmetric space
$\widetilde{X}=\Sl_2(\C)/\SU(2)$ 
satisfies $\delta(\widetilde{X})=1$. Moreover, $\widetilde{X}$ is isometric to
the 3-dimensional hyperbolic 
space $\mathbb{H}^3$. Let $F=\Q(\sqrt{-D})$,
$D\in\mathbb{N}$ square-free, be an imaginary quadratic
number field and let $\mathcal{O}_D$ be its ring of integers. Let
$\Gamma(D):=\Sl_2(\mathcal{O}_D)$. This 
is a discrete subgroup of $\Sl_2(\C)$. It is not cocompact, but the covolume
$\vol(\Gamma(D)\backslash\mathbb{H}^3)$ is finite. If $\aL$ is a
non-zero ideal in $\mathcal{O}_D$, we let $\Gamma(\aL)$ be the
associated 
principal congruence subgroup of level $\aL$. This is a finite
index subgroup of $\Gamma(D)$ which is torsion-free for $N(\aL)$ sufficiently
large. A subgroup $\Gamma_0$ of $\Gamma(D)$ is called a
congruence subgroup if it contains $\Gamma(\aL)$ for some non-zero ideal $\aL$. 

For
$m\in\mathbb{N}$
we let $\rho(m)$ denote the $2m$-th symmetric power of the standard
representation of 
$\Sl_2(\C)$ on $V(m):=\Symm^{2m}\C^2$. Let $V_{\R}(m)$ be $V(m)$ regarded as a
real vector
space and let $\rho_\R(m)$ be the representation $\rho(m)$ regarded as a  real
representation. There exists a canonical lattice $L(m)$ 
in $V_\R(m)$ which is invariant under $\rho_\R(m)(\Gamma(D))$, see section
\ref{secongr}. Then $L(m)$ 
can be regarded as an arithmetic $\Gamma_0$-module. Actually, we work with a
slightly symmetrized 
version of these lattices. Let $L^*(m)$ be the dual lattice of $L(m)$ in
$V^*_\R(m)$.
Then we consider the lattice
$\overline{L}(m):=L(m)\oplus L^*(m)$ in the space $V_\R(m)\oplus V^*_\R(m)$
on which $\Sl_2(\C)$ acts with the representation
$\rho_\R(m)_\R\oplus\check{\rho}_\R(m)$, where
$\check{\rho}_\R(m)$ denotes 
the contragredient representation of $\rho_\R(m)$, which in our case is in fact 
equivalent to $\rho_\R(m)$. 

The main result of this article 
is that the size of the torsion part in the first integral homology with
coefficients 
in $\overline{L}(m)$ grows exponentially for certain sequences of congruence
subgroups of
Bianchi groups. 
More precisely, we will prove the following theorem.
\begin{thrm}\label{Mainthrm}
Let $\Gamma_0\subset\Gamma(D)$ be a
congruence subgroup and assume that $\Gamma_0$ is torsion-free. Let $\Gamma_i$,
$i\in\mathbb{N}$ be a sequence of congruence 
subgroups of $\Gamma(D)$, contained in
$\Gamma_0$, with
\begin{align}\label{AssSeq}
\lim_{i\to\infty}\frac{\kappa(\Gamma_i)\log{[\Gamma(D):\Gamma_i]}}{[
\Gamma(D):\Gamma_i]}=0, 
\end{align}
where $\kappa(\Gamma_i)$ denotes the number of cusps of $\Gamma_i$. Let
$X_0:=\Gamma_0\backslash\mathbb{H}^3$. 
Then for $m\in\mathbb{N}$ with $m\geq 3$ one has
\begin{align}\label{EqMainthrm}
\lim\inf_{i\to\infty}\frac{\log|H_1(\Gamma_i:\overline{L}(m))_{tors}|}{[
\Gamma_0:\Gamma_i]}\geq
\frac{2m(m+1)-12}{\pi}\vol(X_0).
\end{align}

\end{thrm}

Our main result Theorem \ref{Mainthrm} can be applied for example to the
standard
exhaustive 
sequences of congruence subgroups. Namely, the following corollary
holds. 
\begin{kor}\label{lastcor}
If $\aL_i$ is a chain of ideals in $\mathcal{O}_D$ such that
$\lim_{i\to\infty}N(\aL_i)=\infty$, then for $m\geq 3$ one has
\begin{align*}
\lim\inf_{i\to\infty}\frac{\log|H_1(\Gamma(\aL_i);\overline{L}(m))_{tors}|}{
\vol(\Gamma(\aL_i)\backslash\mathbb{H}^3)}\geq \frac{2m(m+1)-12}{\pi}.
\end{align*}
In particular, for 
any congruence subgroup $\Gamma_0$ of $\Gamma(D)$ there exists a
sequence $\Gamma_0\supset\Gamma_1\supset\dots$ 
with $[\Gamma_0:\Gamma_i]<\infty$, $\cap_{i}\Gamma_i=1$ and 
which satisfies \eqref{EqMainthrm}.
\end{kor} 

We shall now briefly describe our approach to prove our main results. 
We let $\Gamma$ be a
congruence subgroup of $\Gamma(D)$ which is torsion-free. Let $\mathbb{H}^3$ be
the hyperbolic 3-space. Then $X:=\Gamma\backslash\mathbb{H}^3$ is a hyperbolic
manifold 
with cusps of finite volume. As already observerd by Bergeron and Venkatesh, the
size of the torsion subgroups $H_q(\Gamma,L(m))_{tors}$ is closely related to
the Reidemeister torsion of $X$ with coefficients in the complex flat vector
bundle
$E(\rho(m))$ defined by the restriction of $\rho(m)$ to $\Gamma$. 
However, in our case the singular homology of this bundle never vanishes. This
is a major 
difference to the cocompact situation and causes additional difficulties:
The Reidemeister torsion with coefficients in $E(\rho(m))$ is 
not an invariant of $X$ and $\rho(m)$ but it depends on the choice of a
particular basis the in the singular homology. To overcome this problem,
Menal-Ferrer 
and Porti introduced the so called normalized Reidemeister torsion of $X$ with 
coefficients in $E(\rho(m))$, see \cite{MePo}. This torsion is defined as a
quotient of 
two Reidemeister torsions for the representations $\rho(m)$ and $\rho(2)$,
$m\geq 3$ with 
respect to bases in the singular homology coming from the boundary of the
Borel-Serre compactification of $X$  
and its definition is independent of the particular choice of such bases, see
\cite[Proposition 2.2]{MePo}. The choice of a particular
basis also enters the 
relation between the Reidemeister torsion and the size of the torsion subgroups
in terms 
of certain volume factors in the free part of the homology.

To prove our main results, in a first step we show that the normalized
Reidemeister 
torsion grows exponentially in certain sequences of coverings. The basic tool we
use is an expression 
of the normalized Reidemeister torsion in terms of special values of Ruelle zeta
functions which has 
been proved by Menal-Ferrer and Porti \cite[Theorem 5.8]{MePo}, who generalized
a
result of M\"uller \cite[equation 8.7]{Muller} to the non-compact 
case. 

In a second step, expressing the normalized Reidemeister torsion as a
weighted product 
of sizes of the groups $|H_q(\Gamma,\overline{L}(m))_{tors}|$,
$|H_q(\Gamma,\overline{L}(2)_{tors}|$  and 
a contribution from the free part in the singular homology, we show that the
latter contribution 
can be controlled in the limit and that the contribution of
$|H_1(\Gamma,\overline{L}(m))_{tors}|$ can 
be isolated in such a way that Theorem \ref{Mainthrm} can be deduced. This 
step covers  
the main part of our proof. \newline

The structure of this paper is as follows. In section \ref{secRuelle} we recall 
the definition of the normalized Reidemeister torsion and prove our main result 
about its asymptotic behaviour. In section \ref{sect} we recapitulate some
facts 
about the homology of arithmetic groups with coefficients in free
$\Z$-modules. In section \ref{secongr} we 
introduce the Bianchi groups and their congruence subgroups. The proof of
Theorem \ref{Mainthrm} and Corollary \ref{lastcor} will be carried out in the
last section
\ref{lastsec}.\newline

At the end of this introduction we want to remark that in parts of his thesis
\cite{Ra}
and
in his recent preprint \cite{Ra2} Jean Raimbault studied the asymptotic
behaviour 
of the regularized analytic torsion and of the Reidemeister torsion for more
general sequences of hyperbolic 3-manifolds and for more general strongly
acyclic coefficient systems. 
This problem is closely related to the problem studied in this paper.

\bigskip
{\bf Acknowledgement.}
I want to thank Pere Menal-Ferrer for interesting
discussions. I also want to thank Jean Raimbault for some useful remarks on
an 
earlier version of this paper.

\section{Ruelle zeta functions and the asymptotics of  normalized Reidemeister
torsion}\label{secRuelle}
Let $X$ be an oriented hyperbolic 3-manifold of finite volume with a fixed
spin-structure. Then there exists a discrete, torsion-free subgroup $\Gamma$ of
$\Sl_2(\C)$ such
that
$X=\Gamma\backslash\mathbb{H}^3$. We assume from now on that $X$ is not compact.
Then
$X$ is of the form 
\begin{align*}
X=C(X)\cup\bigsqcup_{i=1}^{\kappa(X)}F_i,
\end{align*}
with $C(X)\cap F_i=\partial C(X)\cap\partial F_i$ and $F_i\cap F_j=\emptyset$
for $i\neq j$. 
Here $C(X)$ is a compact smooth manifold with boundary $\partial C(X)$
and
each $F_i$, a cusp of $X$, is diffeomorphic to $[1,\infty)\times
T^2$, $T^2$ the 
2-dimensional torus. Moreover, on each $F_i$ the hyperbolic metric of $X$
restricts to the warped 
product metric $y^{-2}dy^2+y^{-2}g_0$, where $g_0$ denotes the suitably
normalized flat 
metric on $T^2$. Thus, gluing at 2-torus to each end of $X$ we obtain a compact
smooth 
manifold $\XX$ with boundary  such that $X$ is the interior of $\XX$ and such
that $\XX$ is homotopy-equivalent 
to $X$.

Let $\mathcal{C}(X)$ be the set 
of non-trivial closed geodesics of $X$. For $c\in\mathcal{C}(X)$ let $\ell(c)$
denote its length. Then there exists a constant $c_X$ such
that 
for each $R\in(0,\infty)$ one has 
\begin{align}\label{estlength}
\#\{c\in\mathcal{C}(X)\colon \ell(c)\leq R\}\leq c_X e^{2R},
\end{align}
see for example \cite[Lemma 4.3, Proposition 4.4]{MePo}. We let 
\begin{align*}
\ell(X):=\ell(\Gamma):= \min\{\ell(c)\colon c\in \mathcal{C}(X) \}
\end{align*}
A closed geodesic $c\in
\mathcal{C}(X)$ 
is called prime if it is the shortest among the closed geodesics having the 
same image as $c$.
The set of prime geodesics will be denoted by $\mathcal{PC}(X)$. If
$c\in\mathcal{C}(X)$, 
there exists a unique prime geodesic $c_0\in\mathcal{C}(X)$ with the same image
as $c$ 
and thus there exists a unique integer $n(c)$ such that
$\ell(c)=n(c)\ell(c_0)$. 

For $\gamma\in\Gamma$ we let $[\gamma]$ be its conjugacy class in $\Gamma$. 
We let $C(\Gamma)_s$ be the set of all conjugacy classes $[\gamma]$ such that  
$\gamma\in\Gamma$ is semisimple and non-trivial. Then the set $C(\Gamma)_s$
corresponds bijectively to $\mathcal{C}(X)$, see for example \cite[section
3]{Pf}. If $\gamma\in\Gamma$ 
is semisimple, we let $c(\gamma)$ be the closed geodesic associated to
$[\gamma]$ and we 
let $\ell(\gamma)$ be its length. Let 
\begin{align*}
M:=\left\{\begin{pmatrix}&e^{i\theta}&0\\&0&e^{-i\theta}\end{pmatrix}
\colon\theta\in[0,2\pi)\right\}; \quad H_1:=\begin{pmatrix}1&0\\
0&-1\end{pmatrix}.
\end{align*}
Then there exists a unique $m_{\gamma}\in M$ such that 
$\gamma$ is $G$-conjugate to $m_\gamma\exp(\ell(\gamma)H_1)$, see for example
\cite[section 3]{Pf}. For
$k\in\frac{1}{2}\mathbb{Z}$ we let $\sigma_k$ be the one-dimensional
representation of 
$M$ on $\C(\sigma_k):=\C$ defined by 
\begin{align*}
\sigma_k\begin{pmatrix} &e^{i\theta}&0\\ &0
&e^{-i\theta}\end{pmatrix}:=e^{2ik\theta}. 
\end{align*}
Now we let $SX$ be the unit-sphere bundle of $X$. Then there is a canonical
isomorphism 
\begin{align*}
SX\cong \Gamma\backslash G/M. 
\end{align*}
Thus every $\sigma=\sigma_k$ defines a locally homogeneous vector bundle
\begin{align*}
V(\sigma):=\Gamma\backslash G\times_\sigma\C(\sigma)
\end{align*}
over 
$SX$. Moreover, the geodesic flow $\Phi$ of $SX$ lifts to a flow $\Phi(\sigma)$
of $V(\sigma)$. Thus for every closed geodesic $c$  
its lift to $V(\sigma)$ defines an endomorphism 
$\mu_\sigma(c)$ on the fibre of
$V(\sigma)$ over $\dot{c}(0)$. Explicitly, if one regards $\mu_\sigma$ as an
endomorphism 
on $\C(\sigma)$, one has $\mu_\sigma(c(\gamma))=\sigma(m_\gamma)$ for every 
semisimple element $\gamma\in\Gamma$. Now for $s\in\C$ with $\Real(s)>2$, the
Ruelle zeta function of $X$ associated to 
the representation $\sigma$ is defined by
\begin{align}\label{Ruelle 1}
R_X(s,\sigma):=\prod_{c\in\mathcal{PC}(X)}\det{\left(\Id-\mu_\sigma(c)e^{
-s\ell(
c)}\right)}.
\end{align}
To establish the convergence of this product, we remarkt that one has
\begin{align}\label{Log RZF}
\log{R_X(s,\sigma)}=&\sum_{c\in\mathcal{PC}(X)}
\log{\left(1-\mu_\sigma(c)e^{-s\ell(c)}\right)}\nonumber\\
=&-\sum_{c\in\mathcal{PC}(X)}\sum_{k=1}^\infty\frac{{(\mu_\sigma(c)
)^{k}}
e^{-ks\ell(c)}}{k}\nonumber\\
=&-\sum_{c\in\mathcal{C}(X)}
\frac{\mu_\sigma(c)e^{-s\ell(c)}}{n(c)}.
\end{align} 

Thus by \eqref{estlength}, the last series converges absolutely for $s\in\C$
with $\Real(s)>2$. We 
remark that if $X$ is compact or if more generally $X$ is of finite volume
and
$\Gamma$ is neat, then by
\cite{Bunke}, \cite{Pf} the function $R(s,\sigma)$ admits 
a meromorphic continuation to $\C$. Under coverings, the Ruelle zeta functions
can be estimated 
as follows.

\begin{lem}\label{Lemcov}
Let $X_0$ be an oriented hyperbolic 3-manifold of finite volume. Then there exists a
constant 
$C(X_0)$ such that for each hyperbolic 3-manifold $X$ which is a finite
covering 
of $X_0$ of index $[X_0:X]$ and every $s\in [3,\infty)$ one can estimate
\begin{align*}
|\log{R_X(s,\sigma_k)}|\leq C(X_0)[X_0:X] e^{-\frac{\ell(X)}{2}}
\end{align*}
\end{lem}
\begin{proof}
Let $\pi:X\to X_0$ be a covering of index $[X_0:X]<\infty$. Let $c$ be a closed
non-trivial geodesic 
in $X$ of length $\ell(c)$. Then $\pi(c)$ is a closed non-trivial geodesic in
$X_0$ of the 
same length. Moreover, there are at most $[X_0:X]$ closed geodesics $c'$ in $X$
projecting to $\pi(c)$. 
Thus by \eqref{Log RZF}, for $s\in [3,\infty)$ one can estimate
\begin{align*}
|\log R_X(s,\sigma_k)| \leq e^{-\frac{\ell(X)}{2}}\sum_{c\in \mathcal{C}(X)}
e^{-(s-\frac{1}{2})\ell(c)}
\leq e^{-\frac{\ell(X)}{2}}
[X_0:X]\sum_{c\in\mathcal{C}(X_0)}e^{-(s-\frac{1}{2})\ell(c)}.
\end{align*}
Thus if we let
\begin{align*}
C(X_0):=\sum_{c\in\mathcal{C}(X_0)}e^{-\frac{5}{2}\ell(c)},
\end{align*}
then $C(X_0)<\infty$ by \eqref{estlength} and the lemma follows. 
\end{proof}

Since $X$ is
homotopy-equivalent to $\XX$, we can
identify $\Gamma$ with the fundamental group of $\XX$. 
Restricting $\rho(m)$ to $\Gamma$, we can form the homology groups
$H_q(\XX;V(m))$ 
with coefficients in 
the complex $\Gamma$-module $V(m)$, see section \ref{sect}. These 
groups are finite-dimensional complex vector spaces. We assume from now on that $m\geq 1$. 

Let $\PX$ be the boundary of $\XX$, let $T_j$, $j=1\dots,\kappa(X)$ be the
components of $\PX$ and for each $j$ let $\iota:T_j\to\XX$ denote 
the inclusion. Then we obtain a map from of the fundamental group $\pi_1(T_j)$
of $T_j$ into
$\Gamma$ and
in 
this way we regard $V(m)$ as a $\pi_1(T_j)$-module. Then the
corresponding homology groups with coefficients will be denoted by
$H_q(T_j,V(m))$. For each $q$ there is a natural map 
$\iota_*:H_q(T_j;V(m))\to H_q(\XX;V(m))$. 

According to Menal-Ferrer and Porti, \cite{MePo}, this map can be used to define
bases in the homology groups $H_1(\XX;V(m))$,
$H_2(\XX;V(m))$ as follows. Fix for each $j$ a non-zero vector
$\omega_j(m)\in V(m)$
fixed under $\rho(m)(\pi_1(T_j))$. The set of such vectors is a
one-dimensional complex vector space, see \cite[Lemma 2.4]{MePo1}. 
Fix moreover for each $j$ a non-trivial cylce $\theta_j\in H_1(T_j;\Z)$ and let 
$\eta_j\in H_2(T_j;\Z)$ be a $\Z$-generator of $H_2(T_j;\Z)$. Then the elements
$\theta_j\otimes\omega_j(m)$
resp. $\eta_j\otimes\omega_j(m)$ are 
elements of $H_1(T_j,V(m))$ resp.
$H_2(T_j,V(m))$ and the following
proposition
holds.
\begin{prop}\label{PropMePo}
The maps $\iota_*:H_2(\PX;V(m))\to H_2(\XX;V(m))$ and $\iota_*:H_1(\PX;V(m))\to
H_1(\XX;V(m))$ 
are surjective. More precisely, let
$\theta:=\{\theta_1,\dots,\theta_{\kappa(X)}\}$, let
$\eta:=\{\eta_1,\dots,\eta_{\kappa(X)}\}$, and let
$\omega(m):=\{\omega_1(m),\dots,\omega_{\kappa(X)}(m)\}$. Then the sets
\begin{align}\label{baseMePo}
\mathcal{B}(\theta;\omega(m)):=\bigsqcup_{j=1}^{\kappa(X)}\{
\iota_*(\theta_j\otimes\omega_j(m))\}; \:
\mathcal{B}(\eta;\omega(m)):=\bigsqcup_{j=1}^{\kappa(X)}
\{\iota_*(\eta_j\otimes\omega_j(m))\}
\end{align}
form a basis of $H_1(\XX;V(m))$ resp. $H_2(\XX;V(m))$.
\end{prop}
\begin{proof}
This is proved by Menal-Ferrer and Porti, \cite[Proposition 2.10]{MePo}.
\end{proof}

Since $\XX$ is homotopy-equivalent to a 2-dimensional CW-complex, 
the group  $H_3(\XX;V(m))$ is trival. Since $V(m)$ is
self-contragredient, one has $H_0(\XX;V(m)\cong H^0(\XX;V(m))$ and 
since the last group coincides with the $\Gamma$-invariant vectors in
$V(m)$, $H_0(\XX;V(m))$ is also trivial. Thus
$\mathcal{B}(\theta,\eta,\omega(m)):=\mathcal{B}(\theta;\omega(m))\sqcup\mathcal
{B}(\eta;\omega(m))$ 
defines a basis in the homology of $\XX$ with respect to the local system 
associated to $\rho(m)$.

Let
$T_{\XX}(\rho(m);\mathcal{B}(\theta,\eta,\omega(m))$ be the
Reidemeister-torsion 
of $\XX$ associated to the restriction of the complex representation $\rho(m)$
to $\Gamma$ and with respect
to the basis 
$\mathcal{B}(\theta,\eta,\omega(m))$, see \cite[section 1]{Muellereins}. 
Then by \cite[Proposition 2.2]{MePo}, for $m\geq 3$ the quotient
\begin{align*}
\mathcal{T}_{\XX}(\rho(m)):=\frac{T_{\XX}(\rho(m);
\mathcal{B}(\theta,\eta,\omega(m))}{T_{\XX}(\rho(2);\mathcal{B}(\theta,\eta,
\omega(2)))}
\end{align*}
is independent of the choice of $\theta$. It is easy to see that the quotient is
also independent of $\omega(m)$, $\omega(2)$ and $\eta$. We remark that since
$\rho(m)$ 
factors through a representation of $\PSL_2(\C)$, 
$T_{\XX}(\rho(m))$ is 
also independent of the spin-structure of $X$. This independence of the
spin-structure 
is the main reason why we restrict to even symmetric powers of the standard
representation in this article. 

We shall now study the asymptotic behaviour of the normalized Reidemeister torsion
under 
sequences of hyperbolic manifolds $X_i$. If the constant
$c_{\mathbb{H}^3}(\rho(m))$ is defined according
to \eqref{L2tors}, one has
\begin{align}\label{ConstL2tors}
c_{\mathbb{H}^3}(\rho(m))=\frac{m(m+1)}{\pi}+\frac{1}{6\pi},  
\end{align}
see \cite[5.9.3, Example 3]{BV}, \cite[Remark 2]{MP1}. 
Moreover, the following proposition holds. 
\begin{prop}\label{PropRT}
Let $X_0$ be an oriented hyperbolic 3-manifold of finite volume. Let
$X_i$, $i\in\N$ be 
a sequence of hyperbolic 3-manifolds which are finite coverings of $X_0$ such
that
$\lim_{i\to\infty}\ell(X_i)=\infty$. 
Then for every $m\in\mathbb{N}$ with $m\geq 3$ one has 
\begin{align*}
\lim_{i\to\infty}\frac{\log{\mathcal{T}_{{\XX}_i}(\rho(m))}}{\vol(X_i)}=-(c_{\mathbb
{
H}^3}(\rho(m))-c_{\mathbb{H}^3}(\rho(2)))=-\frac{m(m+1)-6}{\pi}.
\end{align*}
\end{prop}
\begin{proof}
By \cite[Theorem 5.8]{MePo}, taking the different parametrizations into account,
we have
\begin{align*}
\log{\mathcal{T}_{{\XX}_i}}(\rho(m))=-\frac{1}{\pi}\vol{(X_i)}(m(m+1)-6)+\sum_
{
k=3}^m\log\left|R_{X_i}(k,\sigma_k)\right|.
\end{align*}
Applying Lemma \ref{Lemcov}, the proposition follows. 
\end{proof}
To explain Proposition \ref{PropRT} we note that for
$m\in\mathbb{N}-\{0\}$ the representation $\rho(m)$ is not invariant 
under the standard Cartan-Involution of $\Sl_2(\C)$. Thus by \cite[Lemma
4.1]{BV}, for closed hyperbolic
3-manifolds 
$X$ the bundle $E(\rho(m))$ over $X$ is strongly acyclic. In particular, by the
Hodge-DeRham isomorphism, in the closed case
the homology groups $H_q(X,E(\rho(m)))$ vanish and 
the Reidemeister torsion $T_X(\rho(m))$ is an invariant of the closed manifold
$X$ and the 
representation $\rho(m)$. Moreover, as a special case of a more general Theorem
of 
Bergeron and Venkatesh \cite[Theorem 4.5]{BV} one obtains for 
a sequence of closed hyperbolic 3-manifolds $X_i$ with $\ell(X_i)\to\infty$ and
for every $m\in\mathbb{N}-\{0\}$ that 
\begin{align*}
\lim_{i\to\infty}\frac{\log{T_{X_i}(\rho(m))}}{\vol(X_i)}=-c_{
\mathbb{H}^3}(\rho(m)).
\end{align*}
Since in the closed case one has
$\log{\mathcal{T}_{X_i}}(m)=\log{T_{X_i}(\rho(m))}-\log{T_{X_i}(\rho(2))}$,
Proposition \ref{PropRT} can be seen as a modified extension of the result
of Bergeron and Venkatesh to the
non-compact 3-dimensional case.

\section{Torsion in the homology of arithmetic groups}\label{sect}
We keep the notation of the previous section. We fix a finite-dimensional
smooth representation $\rho$ of 
$\Sl_2(\C)$ on a complex vector-space $V$. We let $V_\R$ be $V$, regarded as a
real vector-space, and 
we we let $\rho_\R$ be the corresponding real representation of $\Sl_2(\C)$.
Assume  
that there exists a lattice $L$ in $V_\R$ which is invariant under
$\rho_\R(\Gamma)$. Then 
we obtain a representation $\rho_\Z:\Gamma\to\Aut_\Z(L)$. 
For 
$q=0,1,2$ we denote the associated homology resp. cohomology groups of
$\Gamma$ with coefficients 
in $L$ by $H_q(\Gamma;L)$ resp. $H^q(\Gamma;L)$. These groups can
be
computed as 
follows. 
Let $K$ be a smooth triangulation of $\XX$, containing a subcomplex $J$
triangulating $\PX$. Let $\tilde{X}$ be the universal covering of $\XX$.
Then $\tilde{X}$ is homotopy equivalent to $\mathbb{H}^3$, the hyperbolic
3-space. In particular, $\tilde{X}$ is contractible. Let
$C_q(\tilde{K})$ be
the free abelian group generated by the $q$-chains of $\tilde{K}$, let 
$C^q(\tilde{K}):=\Hom_{\Z}(C_q(\tilde{K}),\Z)$ and 
let $C_*(\tilde{K})$ resp. $C^*(\tilde{K})$ be the associated simplical
chain-
resp. cochain-complexes. Each
$C_q(\tilde{K})$ is a free $\Z[\Gamma]$ module and if one fixes an embedding
of 
$K$ into $\tilde{K}$, then the $q$-cells of $K$ form a basis of
$C_q(\tilde{K})$
over $\Z[\Gamma]$. Let  
\begin{align*}
C_q(K;L):=C_q(\tilde{K})\otimes_{\Z[\Gamma]}L;\quad
C^q(K;L):=C^q(\tilde{K})\otimes_{\Z[\Gamma]}L. 
\end{align*}
Then the $C_q(K;L)$ resp. $C^q(K;L)$ again form a chain resp. a
cochain
complex $C_*(K;L)$ resp. $C^*(K;L)$ of free $\Z$-modules of finite rank  and 
the corresponding homology resp. cohomology groups, which are topological
invariants of 
$\XX$, will be denoted by
$H_q(\XX; L)$ resp. $H^q(\overline X; L)$. Moreover, since $\tilde{X}$
is contractible, the $C_q(\tilde{K})$ 
form a free resolution of $\Z$ over $\Z[\Gamma]$ and thus one has isomorphisms
\begin{align*}
H_q(\Gamma;L)\cong H_q(\overline X;L);\quad H^q(\Gamma;L)\cong
H^q(\XX;L),
\end{align*}
where the second isomorphism follows from the isomorphism
\begin{align*}
C^q(K,L)\cong\Hom_{\Z[\Gamma]}(C_q(\tilde{K}),L),
\end{align*}
which induces an isomorphism of the corresponding cochain complexes.

Let $\tilde{J}$ be the subcomplex of $\tilde{K}$ coming from the subcomplex
$J$. 
Restricting $\rho_{\Z}$ to the image of $\pi_1(\PX)$ in $\Gamma$ and using
$\tilde{J}$, we can
form
the complex $C_*(J;L))$ 
which is a subcomplex of $C_*(K;L)$. The homology groups of $C_*(J;L)$
are
topological 
invariants of $\PX$ and will be denoted by $H_q(\PX;L)$.  Finall by
$H_q(\XX,\PX;L)$ we denote
the relative homology, i.e. the homology of the complex
$C_*(K;L)/C_*(J;L)$.

If we denote by $A$ one of the homology resp. cohomology groups, then $A$ is 
a finitely generated abelian group and thus it has a decomposition
$A=A_{free}\oplus A_{tors}$, where $A_{free}$ is a finite-rank free $\Z$-module
and where $A_{tors}$ is a
finite abelian group. Now we let $L^*:=\Hom_\Z(L,\Z)$. Then $L^*$ becomes a
$\Gamma$-module via the contragredient representation $\check{\rho}_\Z$ of
$\rho_\Z$.
Moreover, we have the following Lemma.

\begin{lem}\label{duality}
For each $q$ there is an isomorphism $H^q(\Gamma;L^*)_{tors}\cong
H_{q-1}(\Gamma;L))_{tors}$.
\end{lem}
\begin{proof}
There is an isomorphism of complexes 
\begin{align*}
C^{*}(K;L^*)\cong \Hom_{\Z}(C_{*}(K;L);\Z)
\end{align*}
and thus the statement follows from the universal coefficient theorem. 
\end{proof}

Now we let 
\begin{align*}
C_q(K,V_\R):=C_q(\tilde{K})\otimes_{\R[\Gamma]}V_\R.
\end{align*} 
The $C_q(K,V_\R)$ form a chain complex $C_*(\tilde{K},V_\R)$ of
finite-dimensional $\R$-vector spaces and
\begin{align}\label{tens}
C_q(K,V_\R)=C_q(K,L)\otimes_\Z\R.
\end{align}
The homology groups 
$H_q(K,V_\R)$ of the complex $C_*(\tilde{K},V_\R)$ are topological invariants of
$\XX$
and are equal to the 
homology groups $H_q(\XX;V_\R)$ of $\XX$ with coefficients in the local
system defined by $\rho_\R$. By \eqref{tens},
$H_q(\XX;L)_{free}$ 
is a lattice in $H_q(\XX;V_\R)$. 

Similarly, if we regard $\rho$ as a complex representation of $\Gamma$, we 
obtain the cohomology groups $H_q(\XX;V)$, which
are complex vector
spaces. 
Regarded as real vector spaces, they are equal to the $H_q(\XX;V_\R)$. 

Now assume that we are given bases
$\mathcal{B}^\R_q$ of
$H_q(\XX;V_\R)$, $q\in\{1,2\}$. Let
$\mathcal{B}^\R=\mathcal{B}_1^\R\sqcup\mathcal{B}^\R_2$.  
We define an inner
product on $H_q(\XX;V_\R)$ 
by declaring $\mathcal{B}_q^\R$ to be an orthonormal basis. By
$\vol_{\mathcal{B}_q^\R
}H_q(\XX;L)_{free}$ we denote the covolume of the lattice
$H_q(\XX;L)_{free}$ with 
respect to this inner product. Then if $T_{\XX}(\rho_\R;\mathcal{B}^\R)$ is
the
Reidemeister torsion of 
$\XX$ associated to this inner product and the local coefficients defined by
${\rho_{\R}}|_{\Gamma}$, the following lemma holds.

\begin{lem}\label{Lemreg}
One has
\begin{align*}
T_{\XX}(\rho_\R;\mathcal{B}^\R)=\frac{|H_0(\Gamma;L)_{tors}|\vol_{\mathcal{B}
_1^\R}
H_1(\XX;L)_{free}}{|H_1(\Gamma;L)_{tors}|\vol_{\mathcal{B}_2^\R}
H_2(\XX;L)_{free}}.
\end{align*}
\end{lem}
\begin{proof}
This is proved by Bergeron and Venkatesh \cite[section 2.2]{BV}.
\end{proof}

We finally have to relate the Reidemeister torsion of the representation
$\rho_\R$
to the Reidemeister torsion of the representation $\rho$, regarded as a
complex 
representation. If $\mathcal{B}_q$, $q\in\{1,2\}$ are bases of
$H_q(\XX;V)$, $\mathcal{B}=\mathcal{B}_1\sqcup\mathcal{B}_2$, 
we define bases $\mathcal{B}^{\R}_q$ of $H_q(\XX;V_\R)$ by
\begin{align}\label{BaseR}
\mathcal{B}^{\R}_q:=\mathcal{B}_q\sqcup\{\sqrt{-1}\omega\colon \omega\in
\mathcal{B}_q\};\quad
\mathcal{B}^{\R}:=\mathcal{B}^{\R}_1\sqcup\mathcal{B}^\R_2. 
\end{align}
Then the following Lemma holds.
\begin{lem}\label{LemMaMu}
For any complex basis $\mathcal{B}$ as above one has 
\begin{align*}
T_X(\rho;\mathcal{B})^2=T_X(\rho_\R;\mathcal{B}^{\R}).
\end{align*}
\end{lem}
\begin{proof}
For the proof one can proceed as in the proof of \cite[Lemma 2.4]{MaM}.
\end{proof}

\section{Congruence
subgroups of Bianchi groups}\label{secongr}
In this section we collect some basic properties of the Bianchi groups and
their congruence subgroups which are needed for our purposes. 
Let us firstly recall the definition of these groups. We let $F:=\Q(\sqrt{-D})$,
$D\in\mathbb{N}$ square-free, be an imaginary
quadratic 
number field and $d_F$ be its class number. Let $\mathcal{O}_D$ be the ring of
integers of $F$, i.e.
$\mathcal{O}_D=\Z+\sqrt{-D}\Z$ if 
$D\equiv 1,2$ modulo $4$, $\mathcal{O}_D=\Z+\frac{1+\sqrt{-D}}{2}\Z$ if
$D\equiv 3$ modulo $4$. We let
$\Gamma(D):=\Sl_2(\mathcal{O}_D)$ be the associated Bianchi-group. Then
$X_D:=\Gamma(D)\backslash\mathbb{H}^3$ 
is of finite volume
\begin{align*}
\vol(X_D)=\frac{|\delta_F|^{\frac{3}{2}}\zeta_F(2)}{4\pi^2},
\end{align*}
where $\zeta_F$ is the Dedekind zeta function of $F$ and $\delta_F$ is 
is the discriminant of $F$, see \cite{Hu}, \cite[Proposition 2.1]{Sa}. 
Let $\aL$ be any nonzero ideal in $\mathcal{O}_D$ and let
$N(\aL)$ denote its norm. 
Then the associated principal congruence subgroup $\Gamma(\aL)$ is defined as
\begin{align*}
\Gamma(\aL):=\left\{\begin{pmatrix}a&b\\
c&d\end{pmatrix}\in\Sl_2(\mathcal{O}_D)\colon a-1\in\aL; d-1\in\aL;
b,c\in\aL\right\}.
\end{align*}
A subgroup $\Gamma$ of $\Gamma(D)$ is called a
congruence subgroup 
if there exists a non-zero ideal $\aL$ in $\mathcal{O}_D$ such that $\Gamma$
contains $\Gamma(\aL)$ as a subgroup of finite index. 
We recall that by \cite[Corollary 5.2]{Bass} the sequence 
\begin{align*}
1\to\Gamma(\aL)\to \Gamma(D)\to\Sl_2(\mathcal{O}_D/\aL)\to 1
\end{align*}
is exact. Thus, arguing exactly as in \cite[Chapter 1.6]{Sh} for the
$\Sl_2(\R)$-case, one
obtains 
\begin{align}\label{eqcusp}
[\Gamma(D):\Gamma(\aL)]=\#\Sl_2(\mathcal{O}_D/\aL)=N(\aL)^3\prod_{\mathfrak{p}
|\aL}\left(1-\frac
{1}{N(\mathfrak{p})^2}\right),
\end{align}
where the product is taken over all prime ideals $\mathfrak{p}$ in
$\mathcal{O}_D$ dividing $\aL$. 
Let $\mathbb{P}^1(F)$ be the one-dimensional projective space of $F$. As usual,
we 
write $\infty$ for the element $[1,0]\in \mathbb{P}^1(F)$.
Then $\Sl_2(F)$ acts naturally on $\mathbb{P}^1(F)$ and by \cite[Chapter 7.2,
Proposition 2.2]{EGM}
one has
$\kappa(\Gamma(D))=\#\left(\Gamma(D)\backslash\mathbb{P}^1(F)\right)$,
$\kappa(\Gamma(\aL))=\#\left(\Gamma(\aL)\backslash\mathbb{P}^1(F)\right)$.
Furthermore, by \cite[Chapter 7.2, Theorem 2.4]{EGM}
one has $\kappa(\Gamma(D))=d_F$. Let $P=MAN$ be the
standard parabolic subgroup 
of $\Sl_2(\C)$, where $M$ is as above and where
\begin{align*}
A=\left\{\begin{pmatrix}\lambda&0\\0&\lambda^{-1}\end{pmatrix},\lambda\in\R,
\lambda>0\right\};\quad
N=\left\{\begin{pmatrix}1&b\\0&1\end{pmatrix},b\in\C\right\}.
\end{align*}
Then $P$ is the stabilizer of $\infty$ in $\Sl_2(\C)$. For each
$\eta\in\mathbb{P}^1(F)$ we fix $B_\eta\in\Sl_2(F)$ with
$B_{\eta}\eta=\infty$. 
Then $P_\eta:=B_{\eta}^{-1}PB_{\eta}$ is  the stabilizer of $\eta$ in
$\SL_2(\C)$. We let 
$N_\eta:=B_{\eta}^{-1}NB_{\eta}$. If $\eta\in\mathbb{P}^1(F)$, we let 
$\Gamma(D)_\eta$ resp. $\Gamma(\aL)_\eta$ be the stabilizer of $\eta$ 
in $\Gamma(D)$ resp. $\Gamma(\aL)$. 

The next Lemma is certainly well known to experts. However, since we could not
find a reference, we 
include a proof here. We let $\mathcal{O}_D^*$ be the group of units 
of $\mathcal{O}_D$, i.e. $\mathcal{O}_D^*=\{\pm 1\}$ for $D\neq 1,3$, 
$\mathcal{O}_D^*=\{\pm 1,\pm\sqrt{-1}\}$ for $D=1$, $\mathcal{O}_D^*=\{\pm
1,\pm\frac{1\pm\sqrt{-3}}{2}\}$ for $D=3$.

\begin{lem}\label{Lemcusp}
Let $\aL$ be an ideal in $\mathcal{O}_D$. Then for $N(\aL)$ sufficiently large
one has
$\kappa(\Gamma(\aL))=d_F\frac{[\Gamma(D):\Gamma(\aL)]}{\#(\mathcal{O}_D^*)N(\aL)
}$.
\end{lem}
\begin{proof}
The group $P$ 
is defined over $F$ and we let $P(F)$ be its $F$-valued points. 
If $\lambda\in\C$ is an
eigenvalue of an element
$\gamma\in\Gamma(D)_{\eta}$, then, since 
$B_\eta\gamma B_\eta^{-1}\subset P(F)$, also $\lambda^{-1}$ is an eigenvalue
and thus one has $\lambda,\lambda^{-1}\in F$ and 
since $\mathcal{O}_D$ is integrally closed, one has 
$\lambda\in\mathcal{O}_D^*$. Moreover, if $\mathcal{O}_D^*\neq\{\pm 1\}$, then 
$D=1,3$ and in this case the class number is one and so $\mathcal{O}_D$ has only
one cusp, so in this case one can assume that 
$B_\eta\in\Sl_2(\mathcal{O}_D)$. Thus in any case one obtains
\begin{align}\label{eqB}
B_\eta\Gamma(D)_\eta B_\eta^{-1}=J(B_\eta\Gamma(D)_\eta B_\eta^{-1}\cap N),\:
J\in\left\{\begin{pmatrix}\alpha&0\\ 0&\alpha^{-1}
\end{pmatrix},\alpha\in\mathcal{O}_D^*\right\}. 
\end{align}
Assume that $D\neq 1,3$. Then, since for every element
$\gamma\in\Gamma(\aL)$ one has $\Tr(\gamma)\in
2+\aL$ and since $\Tr(g)=-2$ for every element $g\in -N$, it
follows 
that for  $-4\notin \aL$, i.e. for $N(\aL)$ sufficiently large, one has $B_\eta
\Gamma(\aL)_\eta B_\eta^{-1}\subset N$. If $D=1,3$, then since $\Gamma(\aL)$ is 
a normal subgroup of $\Gamma(D)$ one has $B_\eta \Gamma(\aL)_\eta
B_\eta^{-1}=\Gamma(\aL)_\infty$ and 
it follows from \eqref{eqB} that $\Gamma(\aL)_\infty\subset N$ for $N(\aL)$
sufficiently large. 

Now for $\eta\in\mathbb{P}^1(F)$, $B_\eta$ as above,
write $B_\eta=\begin{pmatrix}\alpha &\beta \\ \gamma &\delta\end{pmatrix}\in
\Sl_2(F)$ and let $\mathfrak{u}$ be the $\mathcal{O}_D$-module 
generated by $\gamma$ and $\delta$. Then one has 
\begin{align*}
B_\eta\Gamma(D)_\eta B_\eta^{-1}\cap N=\left\{\begin{pmatrix}1&\omega\\
0&1\end{pmatrix};\: \omega\in \mathfrak{u}^{-2}\right\};\:
B_\eta\Gamma(\aL)_{\eta}B_\eta^{-1}\cap N=\left\{\begin{pmatrix}1&\omega'\\
0&1\end{pmatrix};\: \omega'\in \aL\mathfrak{u}^{-2}\right\},
\end{align*}
where the first equality is proved in \cite[Chapter 8.2, Lemma 2.2]{EGM} and
where the second equality
can be proved using the same arguments. 
Thus one has 
\begin{align*}
[B_\eta\Gamma(D)_\eta B_\eta^{-1}\cap N:B_\eta\Gamma(\aL)_\eta B_\eta^{-1}\cap
N]=N(\aL).
\end{align*}
Thus by \eqref{eqB}, for each $\eta\in \mathbb{P}^1(F)$ and $N(\aL)$ suffciently
large 
one has $[\Gamma(D)_\eta:\Gamma(\aL)_\eta]=\#(\mathcal{O}_D^*)N(\aL)$ and so, if
$\eta_1,\dots,\eta_{d_F}$ denote fixed representatives of
$\Gamma(D)\backslash \mathbb{P}^1(F)$ one obtains 
\begin{align*}
\kappa(\Gamma(\aL))=\#(\Gamma(\aL)\backslash
\mathbb{P}^1(F))=\#\left(\bigsqcup_{i=1}^{d_F}
\Gamma(\aL)\backslash\Gamma(D)/\Gamma(D)_{
\eta_i}\right)=d_F\frac{[\Gamma(D):\Gamma(\aL)]}{\#(\mathcal{O}_D^*) N(\aL)}.
\end{align*}
\end{proof}

We shall now describe the lattices $L(m)$ in $V_\R(m)$. Let $e_1$, $e_2$ be the 
standard basis of $\C^2$. Then if we realize $V(m)$ as the space of homogeneous
polynomials
in $e_1$ and $e_2$ of degree $2m$, a complex basis of $V(m)$ is given by
$\{e_1^{2m-i}e_2^{i},i=0,\dots,2m\}$. Moreover, for each
$g\in\Sl_2(\mathcal{O}_D)$ 
the matrix representing $\rho(m)(g)$ with respect to this basis has entries in
$\mathcal{O}_D$. 
Thus if we let $L(m)$ be the $\Z$-module spanned
by $\{e_1^{2m-i}e_2^{i},i=0,\dots,2m\}$ and
$\{\sqrt{-D}e_1^{2m-i}e_2^{i},i=0,\dots,2m\}$ for $D\equiv 1,2$ modulo 4 resp. 
$\{\frac{1+\sqrt{-D}}{2}e_1^{2m-i}e_2^{i},i=0,\dots,2m\}$ for $D\equiv 3$
modulo
4 , 
then $L(m)$ is a lattice in $V_\R(m)$ which is preserved by
$\rho_\R(m)(\Gamma(D))$. We shall denote the associated representation of
$\Gamma(D)$ on 
$L(m)$ by $\rho_\Z(m)$. In order to make the basis considered by Menal-Ferrer
and Porti into a 
basis consisting of integral elements, we need the following lemma.

\begin{lem}\label{Leminv}
Let $B_\eta\in\Sl_2(F)$ and let $P_\eta:=B_\eta^{-1}PB_\eta$, where $P=MAN$ is
the standard parabolic subgroup 
of $\Sl_2(\C)$ as above. Let $N_\eta:=B_\eta^{-1}NB_\eta$. Then there exist 
vectors $\omega(m)$, $\omega'(m)\in
L(m)$,  
which are fixed by $\rho_{\Z}(m)\left(\pm (\Gamma(D)\cap N_\eta)\right)$ and
which are linearly independent over $\R$. 
\end{lem}
\begin{proof}
The vectors $\omega(m):=e_1^{2m}$, $\omega'(m):=\sqrt{-D}e_1^{2m}$ for $D\equiv
1,2$ modulo 4 resp. $\omega'(m):=\frac{1+\sqrt{-D}}{2}e_1^{2m}$ for $D\equiv 3$
modulo 4 
belong to $L(m)$, are linearly
independent
over $\R$ and are fixed
by $\rho_{\Z}(m)\left(\pm (\Gamma(D)\cap N)\right)$. The matrix 
representing 
$\rho_\R(m)(B_\eta^{-1})$ with respect to a base of $V_\R(m)$ coming from a 
$\Z$-base
of $L(m)$ has entries in $\Q$. Thus multiplying
$\rho_\R(m)(B_\eta^{-1})\omega(m)$,
$\rho_\R(m)(B_\eta^{-1})\omega'(m)$ by the denominator of this matrix 
gives 
the desired vectors for $N_\eta$. 
\end{proof}

\section{Proof of the main results}\label{lastsec}
In this section we prove our main results. We start with the 0-th homology
group. 
This group is pure torsion. Moreover, for sequences of congruence
subgroups the asymptotic behaviour of its size was estimated directly by
Raimbault. 

\begin{prop}\label{PropHzero}
Let $\Gamma_i$ be a sequence of
congruence subgroups of $\Gamma(D)$ 
such that $\lim_{i\to\infty}[\Gamma(D):\Gamma_i]=\infty$. Then for
every $m\in\mathbb{N}$ 
with $m\geq 1$ one has 
\begin{align*}
\lim_{i\to\infty}\frac{\log|H_0(\Gamma_i;L(m))|}{[\Gamma(D):\Gamma_i)]
}=0.
\end{align*}
\end{prop}
\begin{proof}
This is proved by Raimbault, see \cite[Lemma 6.11]{Ra}.
\end{proof}

Now we have to estimate the volume factors with respect to the bases in the
integral 
homology given by Menal-Ferrer and Porti which occur in Lemma \ref{Lemreg}. For
the moment, we 
consider any $X$ of the form $X=\Gamma\backslash \mathbb{H}^3$, where $\Gamma$
is a
discrete, torsion-free
subgroup of $\Gamma(D)$ of finite index. Let $\PX$ be the boundary 
of $\XX$ and let $\iota:\PX\to\XX$ be the inclusion. Let
$\pr:H_1(\XX;L(m))\to
H_1(\XX;L(m))_{free}$ 
be the projection. Recall that $H_2(\XX;L(m))$ is free. It follows from
Proposition \ref{PropMePo} that
$\iota_*H_2(\partial
\overline X;L(m))$ resp. 
$\pr(\iota_*H_1(\partial
\overline X;L(m))_{free})$ are lattices of finite index
$[H_2(\XX;L(m)):\iota_*H_2(\partial
\overline X;L(m))]$ resp. $[H_1(\XX;L(m))_{free}:\pr(\iota_*H_1(\partial
\overline X;L(m))_{free})]$ 
in  $H_2(\XX;L(m))$ resp. $H_1(\XX;L(m))_{free}$.
These indices can be estimated as follows.
\begin{lem}\label{Lemreg2}
One can estimate 
\begin{align*}
[H_2(\XX;L(m)):\iota_*H_2(\partial
\overline X;L(m))]\leq |H_0(\XX;L^*(m))_{tors}|.
\end{align*}
Moreover, one can estimate
\begin{align*}
[H_1(\XX;L(m))_{free}:\pr(\iota_*H_1(\partial
\overline X;L(m))_{free})]\leq |H_1(\XX;L^*(m))_{tors}|.
\end{align*}
\end{lem}
\begin{proof}
We prove the first estimate. The second estimate can be proved in the same way.
Since the sequence 
\begin{align*}
H_2(\PX;L(m))\to H_2(\XX;L(m))\to H_2(\XX,\PX;L(m)) 
\end{align*}
is exact and since $\iota_*H_2(\partial
\overline X;L(m))$ is a lattice of finite index in $H_2(\XX;L(m))$
by Proposition \ref{PropMePo}, the quotient
$H_2(\XX;L(m))/\iota_*H_2(\partial
\overline X;L(m))$ embeds into $H_2(\XX,\PX;L(m))_{tors}$. 
By Poincar\'e duality, \cite[page 223-224]{Wa} one has 
\begin{align*}
H_2(\XX,\PX;L(m))\cong H^1(\XX;L(m)).
\end{align*}
By Lemma \ref{duality} one has 
\begin{align*}
H^1(\XX;L(m))_{tors}\cong H_0(\XX;L(m)^*)_{tors} 
\end{align*}
and the first estimate follows. 
\end{proof}

Now we come to the covering situation. We let $X_0$ be a fixed hyperbolic
manifold of the form
$X_0=\Gamma_0\backslash\mathbb{H}^3$, 
where $\Gamma_0$ is a torsion-free subgroup of $\Gamma(D)$ of finite index. 
We let $\kappa(X_0)$ be the number of cusps of $X_0$ and we let 
$T_{0,1},\dots,T_{0,\kappa(X_0)}$ be the boundary components of $\XX_0$. 
For each
$k=1,\dots,\kappa(X_0)$ we fix, according to Lemma \ref{Leminv}, 
vectors
$\omega_{0,k}^1(m)$, $\omega_{0,k}^2(m)$, in $L(m)$  
which are linearly independent over $\R$ and fixed
by $\rho_\R(m)(\pi_1(T_{0,k}))$. We let 
$X$ be a finite covering of $X_0$. We let 
$\kappa(X)$ be the number of cusps of $X$ and we let $T_j$,
$j=1,\dots,\kappa(X)$ be 
the boundary components of $\XX$. Each $T_{j}$ covers a 
single boundary  component $T_{0,k(j)}$, $1\leq
k(j)\leq \kappa(X_0)$ of $\XX_0$. We let $[T_{0,k(j)}:T_j]$ be the index of this
covering. Moreover, we let
$\omega^1_j(m):=\omega^1_{0,k(j)}(m)$, $\omega^2_j(m):=\omega^2_{0,k(j)}(m)$.

\begin{lem}\label{LemReg2}
Let $m\geq 1$. There exists a constant $C>0$ such that for each hyperbolic
manifold $X$
which 
is a finite covering of $X_0$ and each $j=1,\dots,\kappa(X)$, there exist
non-trivial
cycles $\theta_{j}\in H_1(T_j;\Z)$
such that the free $\Z$-submodule 
\begin{align*}
&\mathcal{M}_{\XX}(m):=\bigoplus_{j=1}^{\kappa(X)}\Z
\pr \iota_*(\theta_{j}\otimes\omega^1_{j}(m))\oplus \Z
\pr \iota_*(\theta_{j}\otimes\omega^2_{j}(m))
\end{align*}
of $H_1(\XX;L(m))_{free}$ 
satisfies
\begin{align*}
[\pr \iota_*(H_1(\PX;L(m))):\mathcal{M}_{\XX}(m)]\leq
C^{\kappa(X)}\prod_{j=1}^{\kappa(X)}[T_{0,k(j)}:T_{j}],
\end{align*}
where $\pr: H_1(\XX;L(m)\to H_1(\XX;L(m))_{free}$ denotes the projection onto
the free part in the homology.
\end{lem}
\begin{proof}
We consider each $H_1(T_{0,k};\Z)$ as a lattice in $H_1(T_{0,k};\R)$ and fix an
inner product defined by a $\Z$-basis of $H_1(T_{0,k};\Z)$. 
The arguments of Menal-Ferrer and Porti \cite[page 17]{MePo} easily imply that
$\iota_*(H_1(T_{0,k};V_\R(m)))$ is a 2-dimensional real vector space 
generated by $\iota_*(\theta\otimes\omega^1_{0,k}(m))$,
$\iota_*(\theta\otimes\omega^2_{0,k}(m))$ for 
each non-zero $\theta\in H_1(T_{0,k};\Z)$.
Let
$v_1(k),v_2(k)\in H_1(T_{0,k};L(m))_{free}$,
such that $\pr \iota_*(H_1(T_{0,k};L(m))_{free})$ is the free $\Z$-module
generated
by $\iota_*(v_1(k)), \iota_*(v_2(k))$. 
Thus for every $\theta\in H_1(T_{0,k};\Z)$ which is not zero there exists 
an integral matrix $A_k(\theta)=(a^k_{\mu,\nu}(\theta))$ such that
in 
the free homology one has 
\begin{align*}
&\iota_*(\theta\otimes\omega^1_{0,k}(m))=a^k_{1,1}(\theta)\iota_*(v_1(k))+a^k_{2
,1}
(\theta)\iota_*(v_2(k)),\\
&\iota_*(\theta\otimes\omega^2_{0,k}(m))=a^k_{1,2}(\theta)\iota_*(v_1(k))+a^k_{2
,
2} (\theta)\iota_*(v_2(k)).
\end{align*}  
Moreover, the matrix $A_k(\theta)$ is invertible over $\R$
for each non-zero $\theta$. 
Since the matrix entries of $A_k(\theta)$ are linear functions of $\theta$,
there exists a constant $C>0$ such that for all $k=1,\dots,\kappa(X_0)$ and all
$\theta\in H_1(T_{0,k};\Z)$ one
has 
\begin{align}\label{estdet}
|\det A_k(\theta)|\leq C|\theta|^2. 
\end{align}
Now fix a boundary component $T_{j}$ of $\PX$. Let $\pi:X\to X_0$ denote 
the covering. Then $\pi$ induces a covering
$\pi:T_{j}\to T_{0,k(j)}$ which in turn induces 
a homomorphism $\pi_*:H_1(T_j;\Z)\to H_1(T_{0,k(j)};\Z)$. By Minkowski's
lattice-point theorem  
there exists  a non-trivial 
$\theta_{j}\in H_1(T_{j};\Z)$ such that 
\begin{align}\label{esttheta}
|\pi_*(\theta_{j})|\leq 2\sqrt{[T_{0,k(j)}:T_j]}.
\end{align}
Again there exist $w_1(j),w_2(j)\in H_1(T_{j};L(m))_{free}$ such that $\pr
\iota_*(w_1(j))$, $\pr \iota_*(w_2(j))$ form a $\Z$-basis of $\pr
\iota_*(H_1(T_{j};L(m))_{free})$. 
Thus in the free homology one has
\begin{align*}
&
\iota_*(\theta_{j}\otimes\omega^1_j(m))=b_{1,1}(\theta_j)\iota_*(w_1(j))+b_{2,1}
(\theta_j)\iota_*(w_2(j))\\
&\iota_*(\theta_{j}\otimes\omega^2_j(m))=b_{1,2}(\theta_j)\iota_*(w_1(j))+b_{2,2
}
(\theta_j)\iota_*(w_2(j)),
\end{align*} 
where the matrix $B(\theta_j)=(b_{\mu,\nu}(\theta_j))$ is integral. By the above
arguments, it is
invertible over $\R$. 

Now we consider 
the map $\pi_*:H_1(\XX,V_\R(m))\to H_1(\XX_0,V_\R(m))$ which restricts to 
a map $\pi_*:H_1(\XX,L(m))_{free}\to H_1(\XX_0,L(m))_{free}$. 
Then we have
\begin{align*}
&\pi_{*}\iota_*w_1(j)=d_{1,1}(j)\iota_*v_1(k(j))+d_{2,1}(j)\iota_*v_2(k(j)),\\
&\pi_{*}\iota_*w_2(j)=d_{2,1}(j)\iota_*v_1(k(j))+d_{2,2}(j)\iota_*v_2(k(j)),
\end{align*}
where $D(j)=(d_{\mu,\nu}(j))$ is an integral $2\times 2$ matrix. We have
\begin{align}\label{eqmat}
A_{k(j)}(\pi_*(\theta_{j}))=D(j)\cdot B(\theta_{j}).
\end{align}
Thus  $D(j)$ is  invertible over $\R$. 
Thus  we have using
\eqref{estdet}, \eqref{esttheta} and  \eqref{eqmat}:
\begin{align*}
&[\pr \iota_*(H_1(T_{j}:L(m))_{free}):\Z
\iota_*(\theta_j\otimes\omega^1_j(m))+\Z \iota_*(\theta_j\otimes\omega^2_j(m))]
=|\det(B(\theta_{j}))|\\ &\leq
|\det(A_{k(j)}(\pi_*\theta_{j}))|\leq 4 C[T_{0,k(j)}:T_{j}]
\end{align*}
and the Lemma follows. 
\end{proof}

Now to treat the image of $H_2(\PX;L(m))$ we will need the following Lemma.
\begin{lem}\label{Lemreg3}
If for each $j$ one chooses a
generator 
$\eta_j$ of $H_2(T_j;\Z)$, the free $\Z$-module 
\begin{align*}
\mathcal{N}_{\XX}(m):=\bigoplus_{j=1}^{\kappa(X)}\Z
\iota_*(\eta_{j}\otimes\omega^1_{j}(m))\oplus \Z
\iota_*(\eta_{j}\otimes\omega^2_{j}(m))
\end{align*}
of $H_2(\XX;L(m))$ satisfies
\begin{align*}
[\iota_*(H_2(\PX;L(m))):\mathcal{N}_{\XX}(m)]\leq
C^{\kappa(X)}\prod_{j=1}^{\kappa(X)}[T_{0,k(j)}:T_{j}].
\end{align*}
\end{lem}
\begin{proof}
The proof is entirely analogous to the proof of the preceding Lemma.  
\end{proof}

Combining Lemma \ref{Lemreg2}, Lemma \ref{LemReg2} and Lemma \ref{Lemreg3}, we
can prove 
the following estimate for the volume factors. 

\begin{prop}\label{PropReg}
Let $\Gamma_0$ be a congruence subgroup of $\Gamma(D)$ and assume that
$\Gamma_0$
is torsion-free.
Let $\Gamma_i$ be a sequence of congruence subgroups of $\Gamma(D)$ contained in
$\Gamma_0$ such that
\begin{align*}
\lim_{i\to\infty}\frac{\kappa(\Gamma_i)\log[\Gamma_0:\Gamma_i]}{[
\Gamma_0:\Gamma_i]}=0.
\end{align*} 
Let $X_i:=\Gamma_i\backslash\mathbb{H}^3$. Let $T_{i,j}$,
$j=1,\dots,\kappa(\Gamma_i)$ denote the boundary components of
$X_i$. 
Then for each $m\in\mathbb{N}$ with $m\geq 1$, each $i$ and each $j$ one can
choose non-trivial cycles $\theta_{i,j}\in
H_1(T_{i,j};\Z)$ and 
vectors $\omega_{i,j}(m)$ fixed by
$\rho_{\Z}(m)(T_{i,j})$ such that
if  
for $\theta_i=\{\theta_{i,1},\dots,\theta_{i,\kappa(\Gamma_i)}\}$,
$\omega_i(m)=\{\omega_{i,1}(m),\dots,\omega_{i,\kappa(\Gamma_i)}(m)\}$ the set
$\mathcal{B}^\R(\theta_i,\omega_i(m))$ is 
the basis of $H_1(\XX_i;\rho_\R(m))$ as in \eqref{baseMePo} and
\eqref{BaseR}, one has
\begin{align*}
0\leq - \frac{\log\vol_{\mathcal{B}^\R(\theta_i,\omega_i(m))}
H_1(\XX_i;L(m))_{free}}{[\Gamma_0:\Gamma_i]}+a_i\leq
\frac{\log{|H_1(\Gamma_i;L^*(m))_{tors}|}}{[\Gamma_0:\Gamma_i]},
\end{align*}
where $a_i\in\R$ with $\lim_{i\to\infty}a_i=0$. Moreover, if for each $i$ and
each $j$ one 
chooses a generator $\eta_{i,j}$ of $
H_2(T_{i,j};\Z)$, then for   
$\eta_i=\{\eta_{i,1},\dots,\eta_{i,\kappa(\Gamma_i)}\}$,
and the basis 
$\mathcal{B}^\R(\eta_i,\omega_i(m))$  
of $H_2(\XX_i;\rho_\R(m))$ defined by \eqref{baseMePo} and
\eqref{BaseR}, one has
\begin{align*}
\lim_{i\to\infty}\frac{\log\vol_{\mathcal{B}^\R(\eta_i,\omega_i(m))}
H_2(\XX_i;L(m))}{[\Gamma_0:\Gamma_i]}=0.
\end{align*}
\end{prop}

\begin{proof}
We choose the $\theta_{i,j}\in H_1(T_{i,j};\Z)$ and the vectors
$\omega_{i,j}^{1}(m), \omega_{i,j}^2\in L(m)$ as in Lemma \ref{LemReg2}. We let
$\omega_{i,j}(m):=\omega_{i,j}^1(m)$.
Let 
\begin{align*}
\widetilde{\mathcal{B}}^\R(\theta_i,\omega_i^{1}(m),\omega_{i}^{2}
(m)):=\bigsqcup_{j=1}^{\kappa(X_i)}\{\iota_*(\omega_
{i,j}^1(m)\otimes\theta_{i,j}),\iota_*(\omega_{i,j}^2(m)\otimes\theta_{i,j})\}
\end{align*}
Then $\widetilde{\mathcal{B}}^\R(\theta_i,\omega_i^{1}(m),\omega_{i}^{2}
(m))$ is a basis of $H_1(X_i,V_\R(m))$ and if $M(i)$ is the matrix of 
base-change from 
$\widetilde{\mathcal{B}}^\R(\theta_i,\omega_i^{1}(m),\omega_{i}^{2}
(m))$ to $\mathcal{B}^\R(\theta_i,\omega_i(m))$ one has
$|\det{M(i)}|=\tilde{C}^{\kappa(\XX_i)}$, where $\tilde{C}\in\R^+$ is a constant
which is 
independent of $X_i$.  
Thus it suffices to estimate the term
$-\log\vol_{\widetilde{\mathcal{B}}^\R(\theta_i,\omega_i^{1}(m),\omega_{i}^{2}
(m))}
H_1(\XX_i;L(m))_{free}$. If $\mathcal{M}_{{\XX}_i}(m)$ is as in Lemma
\ref{LemReg2} one has 
\begin{align*}
\vol_{\widetilde{\mathcal{B}}^\R(\theta_i,\omega_i^{1}(m),\omega_{i}^{2}
(m))}
H_1(\XX_i;L(m))_{free}=\frac{1}{[H_1(\XX_i:L(m))_{free}:\mathcal{M}_{{\XX}_i}
(m)]
}.
\end{align*}
By Lemma \ref{Lemreg2} and Lemma \ref{LemReg2} one can
estimate 
\begin{align*}
1&\leq
[H_1(\XX_i:L(m))_{free}:\mathcal{M}_{{\XX}_i}(m)]\\=&[H_1(\XX_i:L(m))_{free}
:\pr(\iota_*(H_1(\PX_i;L(m))_{free}))][\pr(\iota_*(H_1(\PX_i;L(m))_{free}
)):\mathcal{M}_
{{\XX}_i}(m)]\\
&\leq |H_1(\XX_i;L^*(m))_{tors}|\cdot
C^{\kappa(\Gamma_i)}\prod_{j=1}^{\kappa(\Gamma_i)}[T_{0,m(j)}:T_{i,j}].
\end{align*}
Thus we can estimate
\begin{align*}
&0\leq
-\log\vol_{\widetilde{\mathcal{B}}^\R(\theta_i,\omega_i^{1}(m),\omega_{i}^{2}
(m))}
H_1(\XX_i;L(m))_{free}\\ &\leq
\log|H_1(\XX_i;L^*(m))_{tors}|+\kappa(\Gamma_i)\log{C}+\kappa(\Gamma_i)\log[
\Gamma(D):\Gamma_i]
\end{align*}
and the first estimate follows. Applying Proposition \ref{PropHzero} and Lemma
\ref{Lemreg3}, the second 
estimate can be proved in the same way. 
\end{proof}

Now we can prove Theorem \ref{Mainthrm}. We let the sequence
$X_i=\Gamma_i\backslash\mathbb{H}^3$ be 
as in Theorem \ref{Mainthrm} resp. the previous proposition. 
We let $V^*_\R(m)$ resp. $V^*(m)$ be the dual spaces of
$V_\R(m)$ 
resp. $V(m)$ and we let ${\check{\rho}}_{\R}(m):\Sl_2(\C)\to V^*_\R(m)$ resp.
$\check{\rho}(m):\Sl_2(\C)\to V^*(m)$ be the 
contragredient representation of $\rho_{\R}(m)$ resp. $\rho(m)$ . 
Then the representation $\check{\rho}(m)$, regarded as a 
real representation, is equivalent to $\check{\rho}_\R(m)$. Moreover,
$\check{\rho}(m)$ is self-contragredient. Thus, if for $\mu\in\{2,m\}$ we 
fix vectors $\omega_i(\mu)\in L(\mu)$ as in the previous proposition and if we
let
$\omega_i(\mu)^*\in L(\mu)^*$ be 
the dual vectors, then we obtain  
bases $\mathcal{B}^\R(\theta_i,\omega_i^*(\mu))$  resp.
$\mathcal{B}^\R(\eta_i,\omega_i^*(\mu))$
of $H_1(\XX_i;V^*_\R(\mu))$ resp. of $H_2(\XX_i;V^*_\R(\mu))$ for 
which the previous proposition continues to hold. Now we let
$\bar{\rho}_\R(\mu):=\rho_\R(\mu)\oplus\check{\rho}_\R(\mu)$ acting on 
$\overline{V}_\R(\mu):=V_\R(\mu)\oplus V_\R^*(\mu)$. Then we obtain 
bases
$\mathcal{B}^\R(\theta_i,
\omega_i(\mu))\sqcup\mathcal{B}^\R(\theta_i,\omega_i^*(\mu))$, 
$\mathcal{B}^\R(\eta_i,
\omega_i(\mu))\sqcup\mathcal{B}^\R(\eta_i,\omega_i^*(\mu))$ of
$H_1(\XX_i,\overline{V}_\R(\mu))$ resp. of 
$H_2(\XX_i,\overline{V}_\R(\mu))$. To save notation, we write
$\mathcal{B}_1(\mu):=\mathcal{B}^\R(\theta_i,
\omega_i(\mu))\sqcup\mathcal{B}^\R(\theta_i,\omega_i^*(\mu))$, 
$\mathcal{B}_2(\mu):=\mathcal{B}^\R(\eta_i,
\omega_i(\mu))\sqcup\mathcal{B}^\R(\eta_i,\omega_i^*(\mu))$.\\
We identify $H_*(\XX_i;L(m))\cong H_*(\Gamma_i;L(m))$. By Proposition
\ref{PropReg}
one has 
\begin{align*}
&2\lim\inf_{i\to\infty}\frac{\log|H_1(\Gamma_i;\overline{L}(m))_{tors}|}{
\vol(X_i)}
\\ &\geq 
\lim\inf_{i\to\infty}\frac{\log|H_1(\Gamma_i;\overline{L}(m))_{tors}
|-\log\vol_{\mathcal{B}_1(m)}
H_1(\Gamma_i;\overline{L}(m))_{free}}{\vol(X_i)}.
\end{align*}
On the other hand, by Proposition \ref{PropReg} and Proposition \ref{PropHzero}
one has 
\begin{align*}
&\lim\inf_{i\to\infty}\frac{\log|H_1(\Gamma_i;\overline{L}(m))_{tors}
|-\log\vol_{\mathcal{B}_1(m)}
H_1(\Gamma_i;\overline{L}(m))_{free}}{\vol(X_i)}\\
&\geq
\lim\inf_{i\to\infty}\biggl(\frac{\log|H_1(\Gamma_i;\overline{L}(m))_{tors}
|-\log\vol_{\mathcal{B}_1(m)}
H_1(\Gamma_i;\overline{L}(m))_{free}-\log|H_0(\Gamma_i;\overline{L}(m))_{tors}|}
{\vol(X_i) }\\
&
+\frac{\log\vol_{\mathcal{B}_2(m)}
(H_2(\Gamma_i;\overline{L}
(m))_{ free}
)-\log|H_1(\Gamma_i;\overline{L}(2))_{tors}|+\log\vol_{\mathcal{B}_1(2)}
H_1(\Gamma_i;\overline{L}(2))_{free}
}{\vol(X_i)}\\
&+\frac{\log|H_0(\Gamma_i;\overline{L}(2))_{tors}|-\log\vol_{\mathcal{B}_2(2)}
(H_2(\Gamma_i;\overline{L}(2))_{free})}{
\vol(X_i)}\biggr).
\end{align*}
By Lemma \ref{Lemreg}, Lemma \ref{LemMaMu} and since $\rho(m)$ is
self-contragredient, for each $m\in\mathbb{N}$, the last
$\lim\inf$ equals 
\[
\lim\inf_{i\to\infty}-4\frac{\log{\mathcal{T}_{{\XX}_i}(\rho(m))}}{
\vol(X_i)}.
\]
If $\lim_{i\to\infty}[\Gamma(D):\Gamma_i]=\infty$, then
$\lim_{i\to\infty}\ell(\Gamma_i)=\infty$ and thus 
by Proposition \ref{PropRT} one has
\begin{align*}
\lim\inf_{i\to\infty}-\log{\mathcal{T}_{{\XX}_i}(\rho(m))}=\frac{m(m+1)-6}{\pi}.
\end{align*}
This proves Theorem \ref{Mainthrm}.
\hfill $\square$

To prove Corollary \ref{lastcor}, we remark that 
by \eqref{eqcusp} and Lemma \ref{Lemcusp} we can estimate
\begin{align*}
\frac{\kappa(\Gamma(\aL_i))\log[\Gamma(D):\Gamma(\aL_i)]}{[
\Gamma(D):\Gamma(\aL_i)]}\leq \frac{3d_F\log{N(\aL_i)}}{N(\aL_i)}.
\end{align*}
and thus Corollary \ref{lastcor} follows from Theorem \ref{Mainthrm}. 

\begin{bmrk}
At the end of this article we want to remark that in our oppinion the assumption
made in Theorem \ref{Mainthrm} that all groups $\Gamma_i$ are contained 
in a torsion-free subgroup $\Gamma_0$ of $\Gamma(D)$ is probably 
unneccessary. We think that Theorem \ref{Mainthrm} and its proof presented here
can be
generalized to any sequence of congruence subgroups $\Gamma_i$ of 
$\Gamma(D)$ satisfying \eqref{AssSeq}. 
The assumption that all $\Gamma_i$ are contained in $\Gamma_0$ was only used in
the proof of Lemma \ref{LemReg2} and Lemma \ref{Lemreg3} where we used the
construction of an
explicit basis in the homology with twisted coefficients for the manifold $X_0$ 
given by Menal-Ferrer and Porti in \cite[Proposition 2.2]{MePo}. We think that 
these results of Menal-Ferrer and Porti can in turn be generalized to the 
space $\Gamma(D)\backslash\mathbb{H}^3$ which is not a manifold but only a good
hyperbolic orbifold of finite volume. 
\end{bmrk}

\end{document}